\documentclass{ws-jktr}
\usepackage{amsmath,amsfonts,amssymb,epsf,color,theorem}

\theoremstyle{plain}

\newtheorem{thm}{Theorem}

\newtheorem{prop}{Proposition}

\begin{document}
\markboth{Vladimir Aleksandrovich Krasnov and Vassily Olegovich Manturov}{Graph-Valued Invariants of Virtual and Classical Links and Minimality Problem}
\title{GRAPH-VALUED INVARIANTS OF VIRTUAL AND CLASSICAL LINKS AND MINIMALITY PROBLEM}
\author{Vladimir Aleksandrovich Krasnov and Vassily Olegovich Manturov}
\address{Peoples' Friendship University of Russia}
%
\maketitle

\begin{abstract}
The Kuperberg bracket is a well known invariant of
classical links. Recently, the second named author and L.H.Kauffman
constructed the graph-valued generalisation of the Kuperberg bracket
for the case of virtual links: unlike the classical case, the invariant
in the virtual case is valued in graphs which carry a significant amount of
information about the virtual knot. The crucial difference between
virtual knot theory and classical knot theory is the rich topology of the
ambient space for virtual knots.

In a paper by M. Chrisman and the second named author,
two-component classical links with one fibred component were
considered; the complement to the fibred component allows one to get
highly non-trivial ambient topology for the other component.

In the present paper, we combine the ideas of the above mentioned
papers and construct the ``virtual'' Kuperberg bracket for
two-component links $L=J\sqcup K$ with one component $(J)$ fibred.
We consider a new geometrical complexity for such links and establish minimality of diagrams in a strong sense. Roughly speaking, every other ``diagram'' of the knot in question contains the initial diagram as a subdiagram. We prove a sufficient condition for minimality in a strong sense where minimality can not be established by means introduced in the paper by M. Chrisman and the second nemed author.
\end{abstract}



\section{Introduction. Basic Definitions}\label{s1}

Classical knots can be encoded by diagrams and Reidemeister moves.
There are increasing moves (those which increase the number of
crossings), decreasing moves and the third Reidemeister move, which
does not change the number of crossings. Among the most desirable
(and na\"{\i}ve) questions, one can ask the following:

\begin{enumerate}

\item Given a diagram $K$ of a classical knot; assume no decreasing
move can be applied to $K$. Is it true that $K$ is minimal with respect to the crossing number?

\item Given two diagrams. How to detect whether the resulting knots
differ by mutation?

\end{enumerate}

Certainly, the answers to these questions are negative:
there are lots of diagrams of the trivial knot which do not admit
any decreasing move, see, e.g. \cite{MI}, usually, one can not judge
about mutations of knots looking at
their diagrams.

Possibly, the most striking example where one can judge about some
properties of knots by looking at their diagrams is the Tait
Conjecture, later, Kauffman-Murasugi-Thistlethwaite Theorem~\cite{MT} which says that an irreducible alternating diagram is
minimal.

Moreover, coupled with yet another striking result (Tait flyping
conjecture proved by Menasco and Thistlethwaite~\cite{MT}) says that
two irreducible alternating diagrams represent the same link if and
only if they can be obtained from one another by a sequence of
flypes (see Fig. \ref{flype}).

\begin{figure}[ht]
\centerline{\includegraphics[scale=.8]{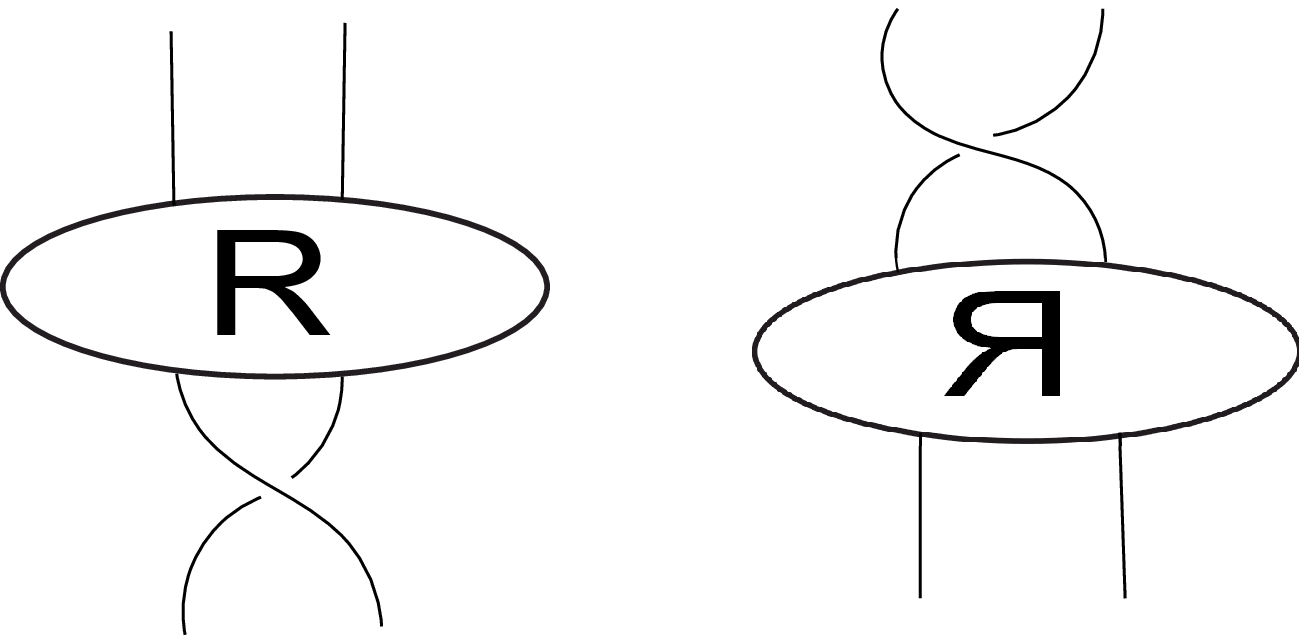}}
\caption{A flype}
\label{flype}
\end{figure}

Note that the proof of the first theorem
(Kauffman-Murasugi-Thisthlethwaite) conjecture is based on one
striking property of one invariant: there is a span of the Kauffman
bracket $sp(K)$ such that $sp(K)\le 4n$ for $n$-crossing
non-split link diagrams, and the equality $sp(K)=4n$ holds for
alternating non-split irreducible link diagrams.

Thus, the Kauffman bracket is useful because one can estimate the
crossing number by using this, but it is hard to say anything about
the shape of the diagram besides of the fact that the diagram is
alternating.

The marvelous result of Menasco and Thislthlethwaite~\cite{MT}
completely solves the classification problem for alternating
classical links but besides some invariants, it also uses various
geometrical arguments.

Nevertheless, the class of alternating links is very small. If we
try to go outside this class then there is still some hope to
estimate the crossing number by using the Kauffman bracket, the
Khovanov homology etc., see \cite{MI}, but as for the shape of
all knot diagrams equivalent to a given one and the complete classification in reasonable
visible geometrical terms, the result seems inapproachable.

Nevertheless, assume we have a knot diagram $K$ schematically shown
in the left part of Fig. \ref{cat}.

\begin{figure}[ht]
\centerline{\includegraphics[scale=.5]{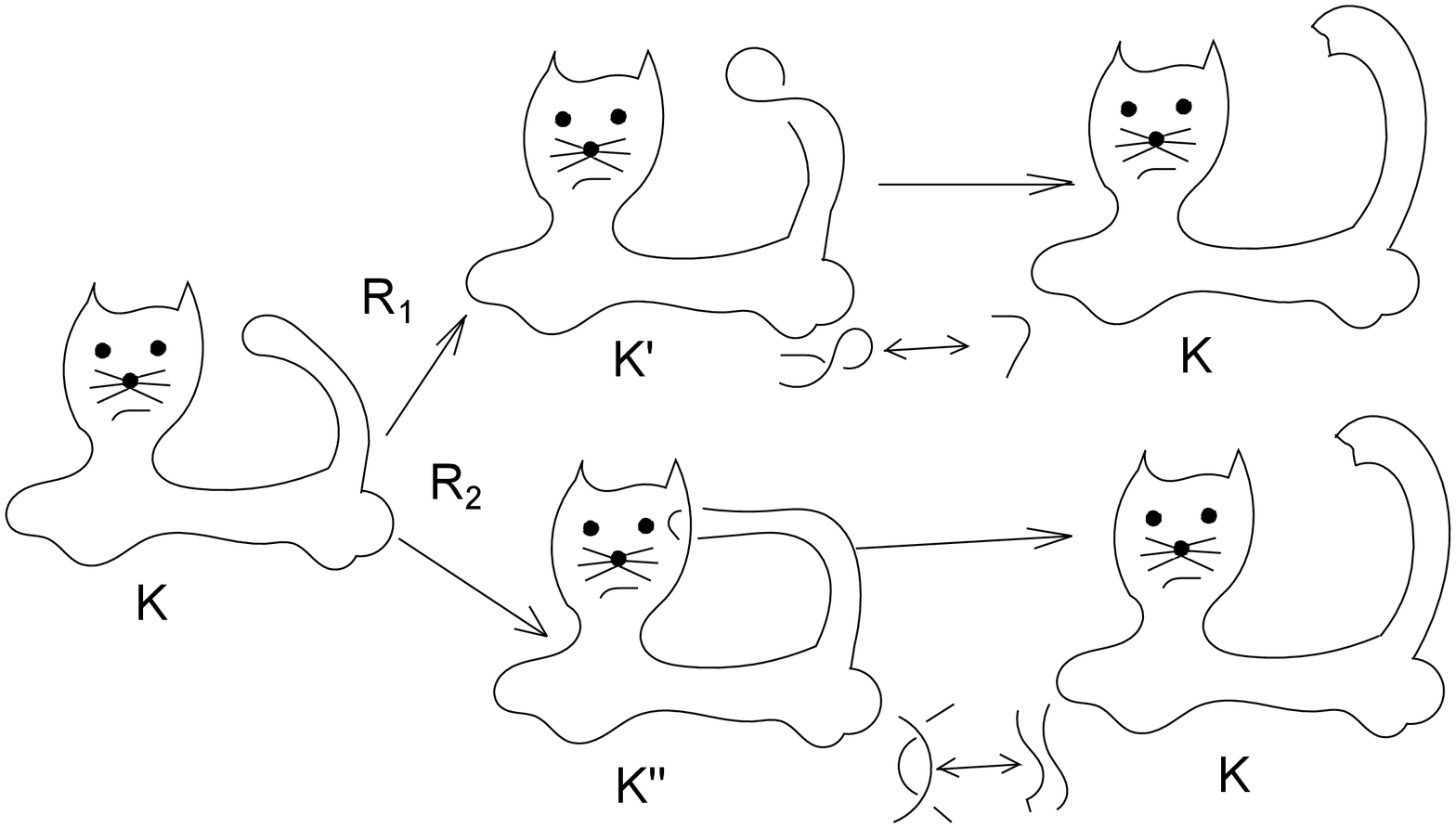}}
\caption{A cat}
\label{cat}
\end{figure}

Assume that it is minimal so that no decreasing move can be applied
to it and assume, furthermore, that no third Reidemeister move can be
applied to it.

Let us start applying Reidemeister moves to $K$. If we apply a first
Reidemeister move to it $K\to K'$, we shall see that $K$ ``sits
inside'' $K'$, i.e., $K$ can be obtained from $K'$ as a result of
''smoothing''. The same can be said about the second increasing move
$K\to K''$.

So, we can see the picture of the initial ``cat'' $K$ inside its
close neighbour diagrams $K'$ and $K''$. Certainly, if we go on
applying Reidemeister moves to $K'$ or $K''$, we can change the diagram in such a way
that no $K$ inside is visible.

But if it were, we could be able to judge about properties of {\em
all} diagrams of the knot represented by $K$ by looking at just one diagram $K$. In some
sense this minimal ``cat'' can be thought of as an invariant of the
knot $K$; it knows much more than just a minimal crossing number: it
has its tail, ears, eyes etc. They all can be seen in any portrait
(diagram) of this knot (cat).

This approach, however, does not work in the classical case: after applying many Reidemeister moves, the initial picture can get lost.

Nevertheless, this program can be partially realised in the case of virutal knots.

The main point is this program can be partially realised in the
virtual domain: there are lots of such ``cats'' which allow one to
judge about properties of {\em any} diagram of a knot by looking at a diagram of this knot (see ~\cite{MI} and~\cite{M}).

M. Chrisman and the second named author started extending this
approach to classical knot theory~\cite{ChM}. In the present paper,
we go on realising the program of extending graphical-valued
invariants to the classical domain.

At first, in the present section, we briefly review the key notions of virtual knot theory.

Knot theory studies embeddings of curves in three-dimensional space. Virtual knot theory
studies embeddings of curves in thickened surface without boundary of arbitrary genus, up to the addition and
removal of empty handles from the surface. Many structures in classical knot theory generalize to the virtual domain.

Virtual knot theory can be thought of a generalisation of classical knot theory from the diagrammatic point of view. We introduce a new
crossing type (a virtual crossing, which should be treated neither as a passage of one branch over the other one nor as a
passage of one branch under the other) and extend new moves to the list of the Reidemeister moves (see Fig. \ref{moves}).

\begin{figure}[ht]
\centerline{\includegraphics[scale=.8]{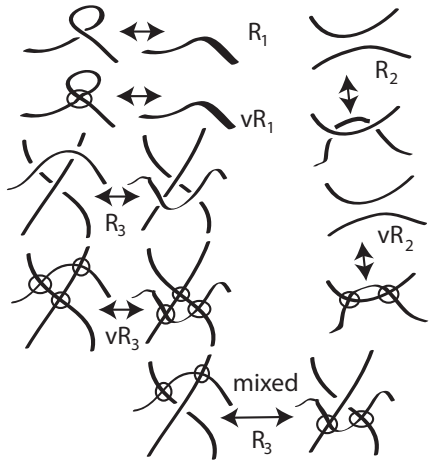}}
\caption{The Reidemeister moves}
\label{moves}
\end{figure}

{\bf Definition 1.} {\em A virtual diagram (or a diagram of a virtual link)} is the image of an immersion of a framed 4-valent graph in $\mathbb{R}^2$ with a finite number of intersections of edges.

{\bf Remark 1.} In Fig. \ref{moves} and what follows whenever depicting any moves/relastions on diagrams, we assume all diagrams to be identical outside a small domain. This small domain where the move takes place will be depicted; the remaining part will be omitted.

{\bf Definition 2.} {\em A virtual link} is an equivalence class of virtual diagrams modulo Reidemeister moves (see Fig. \ref{moves}).

Like classical links, virtual links consists of components. A virtual knot is a one-component virtual link.

Virtual links have a topological interpretation as curves in thickened compact surfaces without boundary. Regard each virtual crossing as a shorthand for a detour of one of the arcs
in the crossing through a 1-handle that has been attached to $\mathbb{R}^2$ of the original diagram. By interpreting each virtual crossing in this way, we obtain an embedding of a collection of circles into a thickened surface $S_{g}\times \mathbb{R}$  where $g$ is the number of virtual crossings in the original diagram
$L$, $S_{g}$ is a compact, connected, oriented surface of genus $g$ and $\mathbb{R}$ denotes the real line (see Fig. \ref{surface}).

\begin{figure}[ht]
\centerline{\includegraphics[scale=.5]{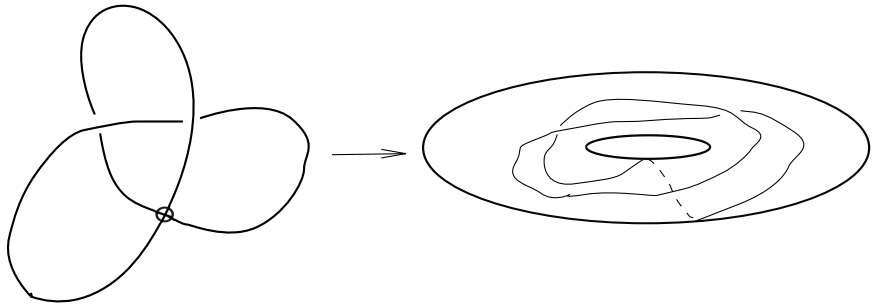}}
\caption{A topological interpretation of the virtual trefoil}
\label{surface}
\end{figure}

Surface embeddings are {\em stably equivalent} if one can be obtained from another by isotopy in the thickened surfaces, homeomorphisms of the surfaces and handle stabilization (see, for example,~\cite{KM}).

\begin{thm}{\bf[see, e.g., \cite{MI} and \cite{Kau}]}
Two virtual link diagrams are equivalent if and only if their corresponding surface embeddings are stably equivalent.
\end{thm}

{\bf Remark 2.} In the present theorem below by equivalene between diagrams we just mean Reidemeister-move equivalence.

\section{The construction of the $sl(3)$ invariant of virtual knots}

Virtual knot theory, due to the non-trivial topology of the ambient space (a thickened 2-surface) has many new striking
invariants of topological and combinatorial nature which do not show up in the classical case. There are many of them
based on the notion of parity due to the second named author.
There are also new invariants of virtual knots due to the second named author and L.H.Kauffman~\cite{KM}; these invariants take values in
certain linear combinations of graphs; the latter, in turn, record a lot of topological information about all diagrams of a given virtual knot or link.
In particular, many properties of virtual links can sometimes be read from one diagram of such a link: e.g., non-invertibility, chirality etc.
When looking at some properties of a given graph which appears in the expansion of the invariant, one can conclude similar properties in
any diagram of the knot from the given class.

An approach undertaken in~\cite{ChM}, the second named author
and M.W.Chrisman have initiated the study of classical two-component links with one fibred component by using virtual knot theory
approach: a cyclic covering over a complement to the fibred component turns out to be a thickened surface, thus,
different combinatorial approaches can be applied. In the present paper, we use the ideas coming from~\cite{KM} and~\cite{ChM}
to construct graphical-valued invariants of 2-component links with one fibred component. In particular, we prove a {\em universal necessary condition of the minimality} of the diagrams of the virtual knots corresponding to the knots in the thickened surfaces. This condition applies for any virtual diagrams unlike result of the work~\cite{ChM} which works only for odd diagram~\cite{M}.

\subsection{A definition of the generalised Kuperberg bracket}
In the present subsection we give a definition of the generalised Kuperberg bracket. This definition will not be as general as it appeared in~\cite{ChM} but it will work both for virtual knots and for knots in a concrete thickened surfaces possibly with a boundary.

Every virtual diagram can be treated as a diagram in a thickened oriented compact 2-manifold. The latter can be with or without boundary. If two diagrams are equivalent in any category (knots in a given surfaces with or without boundary, or virtual knots), then they can be connected by a sequence of Reidemeister moves and vice versa. Thus, when we prove the equivalence of some functions on diagrams under generalised Reidemeister moves, we get an invariant in either category. Certainly, one can take the boundary consideration into account. We shall address this question in a separate paper. Nevertheless, we make some simplifications of our invariant on purpose. In particular, the graphs we are going to consider can be treated as graphs in surfaces (with or without boundary) which can make the structure more powerful. Nevertheless, our goal is to show the main effect that the generalised Kuperberg bracket can easily establish minimality of two-component classical links; to this end, it is quite unnecessary to take care of the whole structure. In a similar manner, the second named author constructed his ``parity bracket''~\cite{M}; in both cases, the main effect is achieved by means of graphs.

From now on, we treat virtual diagrams as diagrams in any category we like.

Let $\mathcal{S}$ be the collection of all connected trivalent
bipartite oriented graphs such that each trivalent vertex of it has either three inward-oriented edges or three outward-oriented edges. Moreover, all graphs in $\mathcal{S}$ are finite, but loops and multiple edges are allowed. Let $\mathcal{T} = \{t_{1},
t_{2},...\}$ be the (infinite) subset of connected graphs from $\mathcal{S}$
having neither bigons nor quadrilaterals. We call this graph having neither bigons nor quadrilaterals {\em a irreducible graph}. Otherwise, we call a graph {\em reducible}.

Let $\mathcal{M}$ be the module $\mathbb{Z}[A,A^{-1}][t_{1}, t_{2}, ...]$ of formal
commutative products of graphs from $\mathcal{T}$ with coefficients
that are Laurent polynomials. Besides, let $\mathcal{D}$ be the
collection of all diagrams of virtual links. Our invariant will be
valued in $\mathcal{M}$.

Below, we collect some results from~\cite{KM}, which we shall need in the sequel.

\begin{thm}{(\bf\cite{KM})}
There is unique map $f:\mathcal{D} \rightarrow \mathcal{M}$ which satisfies the relations $(I)-(VI)$ in Figure \ref{skobka}.
\end{thm}

\begin{figure}[ht]
\centerline{\includegraphics[scale=.7]{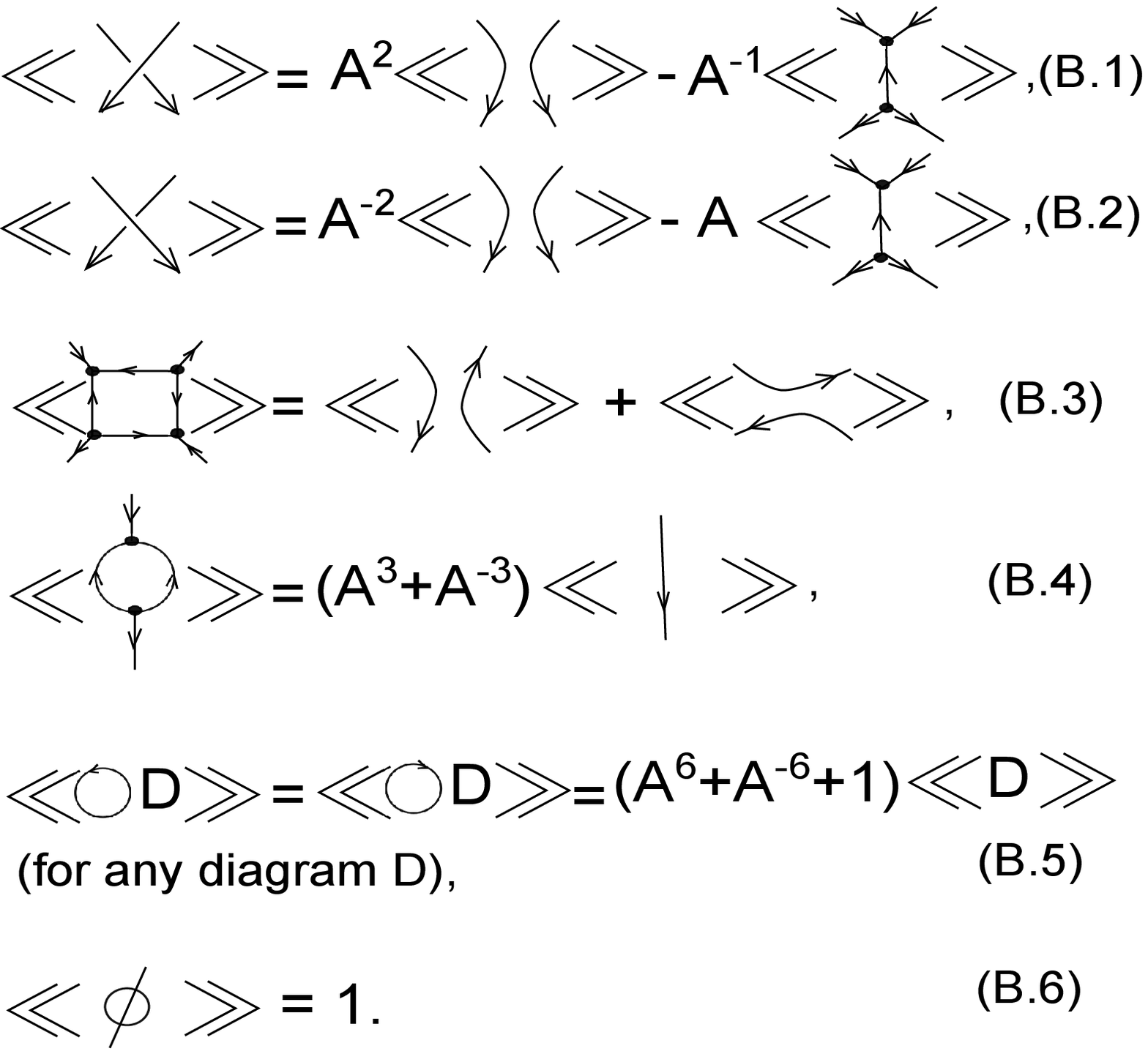}}
\caption{Skein relations for the Kuperberg bracket}
\label{skobka}
\end{figure}

If $L$ be a diagram of the virtual links, then in the sequel we will call its image under the map $f$ {\em a Kuperberg bracket (or non-generalised Kuperberg bracket)} and denote as follows: $\ll L \gg := f(L)$.

{\bf Remark 3.} The Kuperberg bracket was first introduced G. Kuperberg~\cite{Kup} from classical links. We note that value of the Kuperberg bracket for a classical link situated in $\mathbb{Z}[A,A^{-1}]$.

We define {\em the bracket $[[\cdot]]$ (generalised Kuperberg bracket)} as follows.

Let $L$ be an oriented virtual diagram. We can define a polynomial $[[L]]$ as follows:

$$
[[L]]=A^{-8w(L)}\cdot \ll L \gg ,
$$
where $w(L)$ be {\em a writhe number} of the diagram $L$. Let us define $w(L)$ as follows. With each classical crossing of $L$ we associate $+1$ or $-1$ as shown in Fig. \ref{cr}. This number is called the local writhe number. Taking the sum of these numbers at all vertices, we get {\em the writhe number} $w(L)$.

\begin{figure}[h]
\begin{minipage}[h]{0.49\linewidth}
\center{\includegraphics[width=0.3\linewidth]{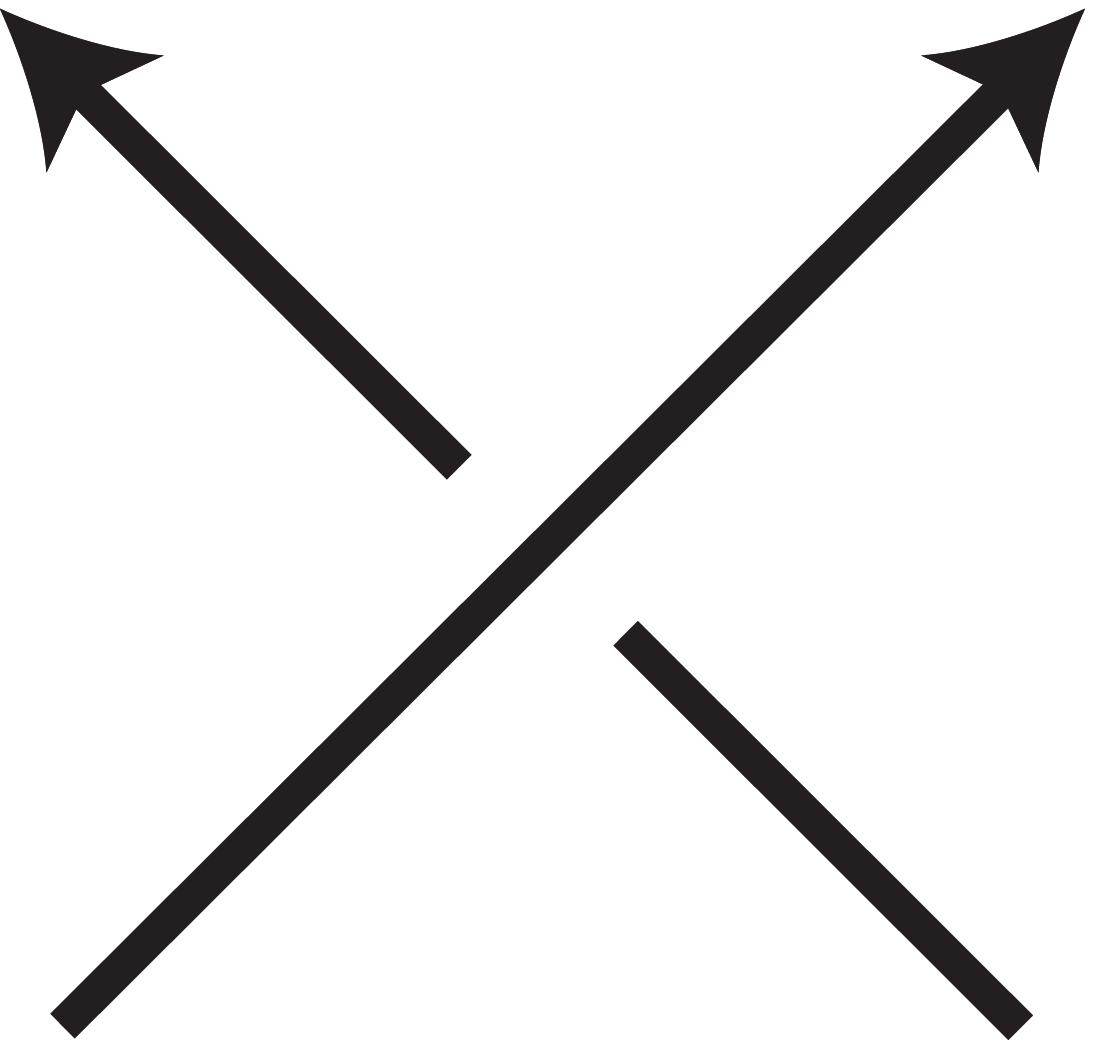} \\ ''+1''}
\end{minipage}
\hfill
\begin{minipage}[h]{0.49\linewidth}
\center{\includegraphics[width=0.3\linewidth]{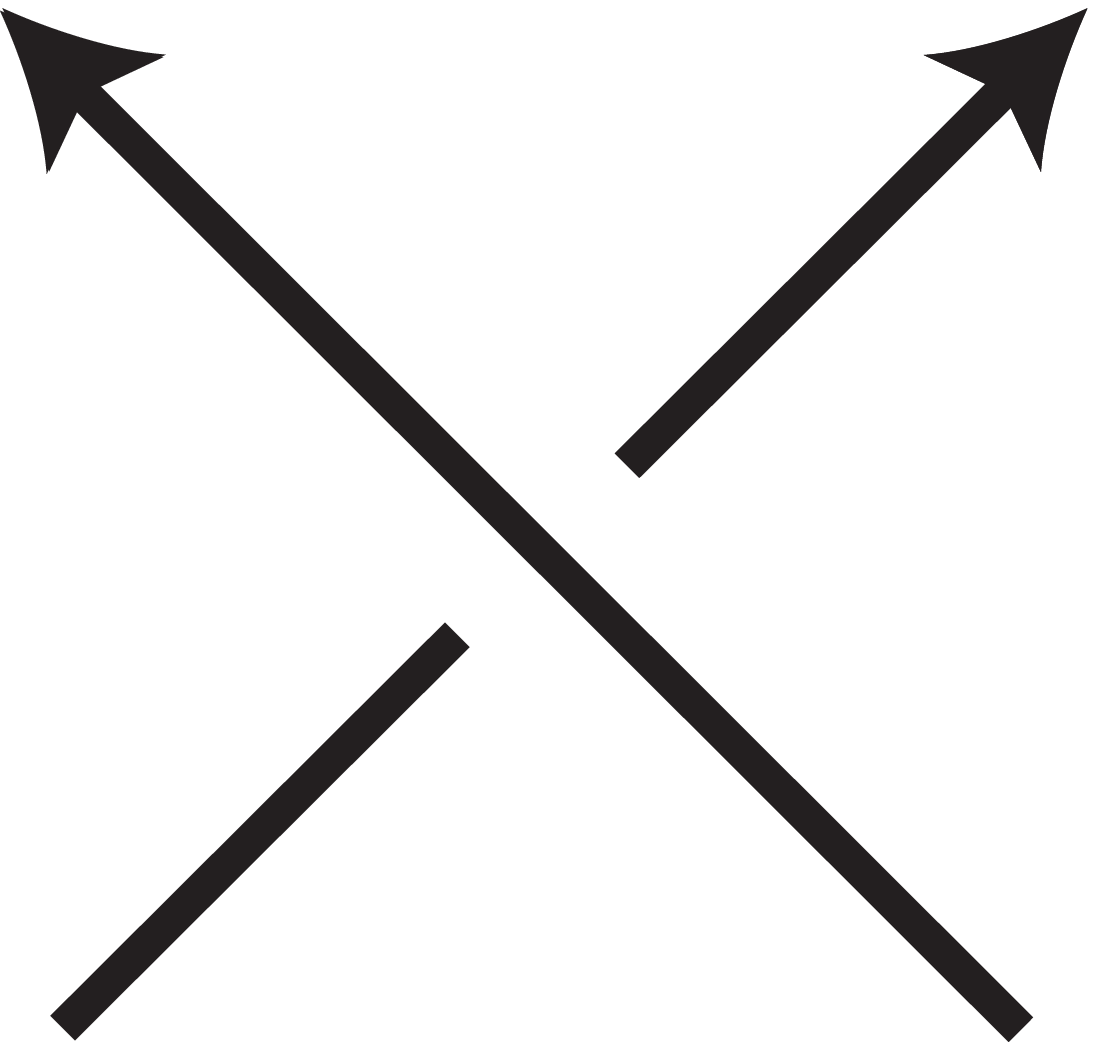} \\ ''-1''}
\end{minipage}
\caption{}
\label{cr}
\end{figure}

\subsection{Basic properties of the generalised Kuperberg bracket}

In the present subsection we collect the some basic properties of the generalised Kuperberg
bracket, which we shall need in the sequel. For proofs, see~\cite{KM}.

\begin{thm}
The bracket $[[.]]$ is invariant under all Reidemeister moves.
\end{thm}

Proof of the invariance of the Kuperberg bracket follows from the
fact that the relations {\em (I)--(VI)} yield invariance under the Reidemeister moves.

We have faced a very important phenomenon:

{\em The relations (I)-(VI) are all reductive, i.e., every graph can
be uniquely represented as a linear combination of irreducible
graphs. Moreover, if a graph itself is irreducible, then it
evaluates to itself. So, if we just dealt with three-valent bipartite
graphs (not knots) we would easily get a desired property we looked
for in the introduction: one can minimise a graph by using some
relation until one gets a linear combination of irreducible graphs. All these irreducible graphs are subgraphs of the initial graph.

Now, if we want to deal with (virtual) knots or knots in any concrete thickened surface with or without boundary, we should first
represent them as linear combinations of graphs by using {\em (I),
(II)}, and then start to simplify the resulting graph.

Note that our graphs are abstract; they do not record any information about their embedding.

The crucial observation is that in the virtual knot theory case
irreducible graphs can appear, unlike the classical case: in the
classical case we deal with planar graphs only by Euler characteristic reasons; such trivalent
bipartite graphs should necessarily have bigons or quadrilaterals
inside; thus, they can not be irreducible!}

The generalised Kuperberg bracket can be also used to solve of the problem of the minimality of virtual diagrams.

Let, as before, $K$ be a virtual knot diagram. With every classical crossing of $K$, we associate two
local states: the oriented one and the unoriented one. The oriented one corresponds to the first term, and the unoriented one corresponds to the second term in the relations {\em (I)--(II)} (see Fig. \ref{skobka}). A state of the diagram is a choice of local state for each crossing in the diagram. Thus, every state is a trivalent bipartite graph. Note that we treat these graphs abstractly, regardless any surface embedding.

By the {\em unoriented state $K_{us}$} ({\em $K_{us}$-graph}) of $K$ we mean the state of $K$ where all crossings are resolved in a way where edge is added (see Figure \ref{kus}).

\begin{figure}[ht]
\centerline{\includegraphics[scale=.3]{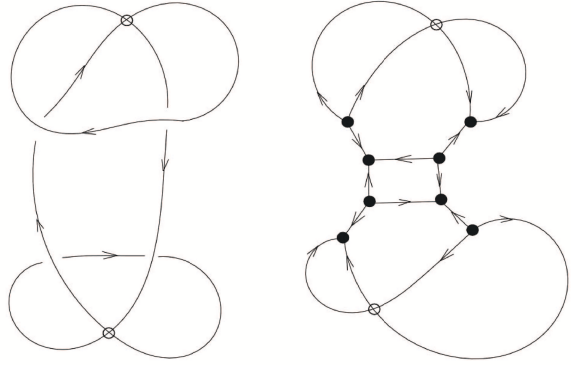}}
\caption{The Kishino knot and its $K_{us}$-state}
\label{kus}
\end{figure}

We note that a $K_{us}$-graph is an important characteristic of the virtual diagram. If $K$ is an $n$-crossing virtual diagram then $K_{us}$ has $2n$ vertices. All other diagrams in the expansion of $[[K]]$ have smaller number of vertices. Thus, after applying $(III)-(VI)$, $K_{us}$ will evaluate to itself, whence all other terms of $[[K]]$ will be represented as linear combinations of graphs with the number of vertices strictly less than $2n$. In this sense, $K_{us}$ can be thought of as the leading term of $[[K]]$, an and for knots with $K_{us}$ irreducible, it is much more convenient to compare their $K_{us}$ than the whole bracket $[[\cdot]]$.

The following proposition can be used to establish the minimality of virtual diagrams~\cite{KM}.

\begin{prop}
Assume for a virtual diagram $K$ with $n$ classical crossings the graph $K_{us}$ is irreducible. Then every knot $K'$ equivalent to $K$ has a state $s$ such that
$K'_{us}$ contains $K_{us}$ as a subgraph. In particular, $K$ is minimal, and all minimal diagrams of this knot have the same number of crossings.
\end{prop}

We now consider some applications of the generalised Kuperberg bracket for a classical two component link $L=J\sqcup K$ such that $lk(L)=0$, when $J$ is fibred.

\section{Some applications of the generalised Kuperberg bracket to study properties of some classical two component links}

\subsection{Fibred knots}

In the present subsection we recall the definition and consider the construction which shall need in the sequel.

{\bf Definition 3.} A classical knot $J$ is {\em fibred} if there a smooth map $b:\mathbb{S}^3\rightarrow D^2=\{z \in \mathbb{C}^2;|z|\leq 1\}$, $0 \in D^2$ is a regular value, $b^{-1}(0)=K$ and a map $\overline{b}:\mathbb{S}^3\backslash J \rightarrow \mathbb{S}^1$, which by the formula $\overline{b}=\frac{b(x)}{|b(x)|}$, is a smooth fibration.

Let $\alpha \in \mathbb{S}^1$ and $[0;\alpha]$ be a segment connecting $0$ and $\alpha$ in $D^2$. Then $b^{-1}([0;\alpha])$ is the Seifert surface of the knot $J$. It is called {\em a fibre} of the fibred knot $J$.

Let $J$ be a knot in $\mathbb{S}^3$ such that the knot complement of $J$
is covered by a thickened surface $\pi:\Sigma \times \mathbb{R} \rightarrow \mathbb{S}^{3} \backslash J$, where $\Sigma$ be a some two-dimension compact orientable surface. Furthermore, let   $L=J\sqcup K$ $lk(L)=0$ be a classical link such that $\pi^{-1}(K)=\bigcup_{i=0}^{\infty} K_{i}$, where each $K_{i}$ is a knot in $\Sigma \times \mathbb{R}$ such that $K_{i}$ is ambient isotopic to $K_{j}$ in $\Sigma \times \mathbb{R}$ for all $i,j$. Links $L=J\sqcup K$ satisfying these properties can also be studied by using virtual knot theory. As there is no ambiguity in the choice of $K_{i}$ (up to isotopy), the stabilization of any $K_{i}$ is a well defined virtual knot $\hat{K_{i}}$.

Now let $J$ be a fibred knot. If the knot $K$ is ''close'' to a knot diagram on a minimal genus Seifert surface $\Sigma$ for $J$, then there is the unique infinite cyclic cover of $\mathbb{S}^3 \backslash J$ of the form $\Sigma \times \mathbb{R}$ relative to a given fibration $\overline{b}$ (see~\cite{ChM} and~\cite{W}). If $K'$ is a knot in $\Sigma \times \mathbb{R}$, let us project it to a fiber $\Sigma$ of the knot complement of $J$. This gives a knot diagram on $\Sigma$. Now push the double points of the projection a little off the Seifert surface in $\mathbb{S}^3$ so that we get a knot $K$ in $\mathbb{S}^3$. $K$ lifts to $K'$ in the infinite cyclic cover and every lift of $K$ is ambient isotopic to $K'$. A link $L=J\sqcup K$ may be studied using virtual knot theory by stabilizing the knot $K'$ in $\Sigma \times \mathbb{R}$ to a virtual knot $\hat{K}$.

It can be shown that the map $g: J\sqcup K \rightarrow \hat{K}$ has a very important property: if $K_1$ and $K_2$ are ambient isotopic by an
isotopy which is the identity in a regular neighborhood of $J$, then virtual diagrams of virtual knots $\hat{K_1}$ and $\hat{K_1}$ are equivalent~\cite{ChM}.

In the next we used the map $g: J\sqcup K \rightarrow \hat{K}$ to study some properties of classical two component links.

\subsection{The generalised Kuperberg bracket and two-component classical links}

We first consider the application of generalized Kuperberg bracket for recognition of links of the form $L=J\sqcup K$ where $lk(L)=0$ such that $J$ is fibred.

Assume $J$ is fixed once forever.

\begin{thm}
Let $L_{1}=J\sqcup K_{1}$ and $L_{2}=J\sqcup K_{2}$ be classical two component links such that $lk(L_{1})=lk(L_{2})=0$, where $J$ be a fibred knot with a Seifert surface $\Sigma$. Furthermore, let $\hat{K_1}$ and $\hat{K_2}$ be virtual diagrams corresponding to $K_{1}'$ and $K_{2}'$ in $\Sigma \times \mathbb{R}$ , where $\pi(K_{1}')=K_{1}$ and $\pi(K_{2}')=K_{2}$. Then if $[[\hat{K_1}]]\neq[[\hat{K_2}]]$ that $K_{1}$ and $K_{2}$ are non ambient isotopic by an smooth isotopy which is the identity in a regular neighborhood of $J$.
\end{thm}

\begin{proof}
Assume the contrary. Supposing $K_{1}'$ and $K_{2}'$ are ambient isotopic by an isotopy which is the identity in a regular neighborhood of $J$. Then $[[\hat{K_1}]]=[[\hat{K_2}]]$ because if $K_{1}$ and $K_{2}$ are ambient isotopic by an isotopy which is the identity in a regular neighborhood of $J$, then virtual diagrams $\hat{K_1}$ and $\hat{K_2}$ are equivalent (see Subsection 3.1.). Therefore, $[[\hat{K_1}]]=[[\hat{K_2}]]$ (see Theorem 3). Contradiction.
\end{proof}

{\bf Example 1.} Let $L_{1}=J\sqcup K_{1}$ and $L_{2}=J\sqcup K_{2}$ be classical two component links such that $lk(J,K)=0$, where $J$ is a fibred knot (for example, right-handed trefoil~\cite{PS}) with the Seifert surface $\Sigma$ of the minimal genus $g=1$ (torus). Also let $\hat{K_1}$ and $\hat{K_2}$ be virtual diagrams corresponding to knots $K_{1}'$ and $K_{2}'$ in $\Sigma \times \mathbb{R}$ (see Fig. \ref{ex1}).

\begin{figure}[ht]
\centerline{\includegraphics[scale=.5]{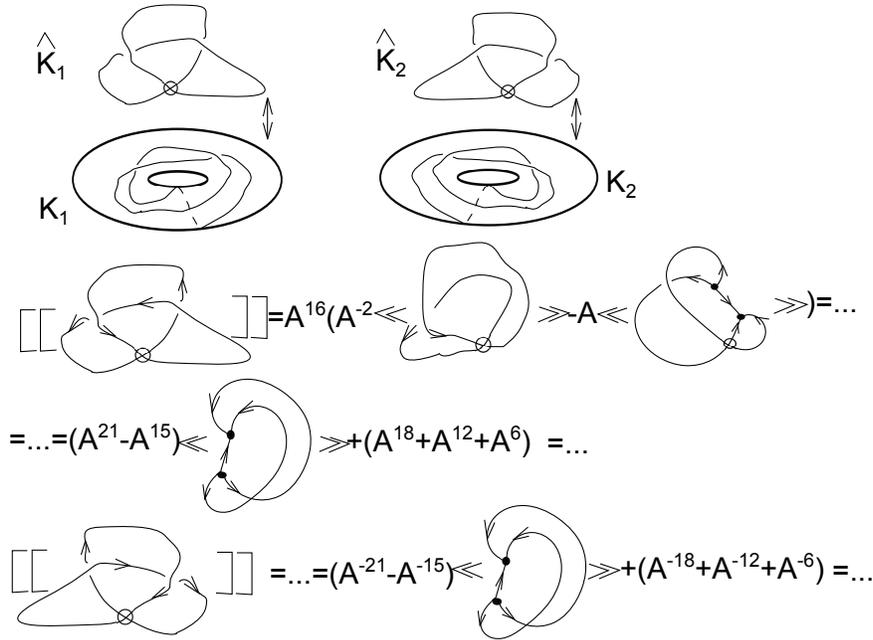}}
\caption{A calculation of the Generalized Kuperberg Bracket for $\hat{K_1}$ and $\hat{K_2}$}
\label{ex1}
\end{figure}

From the further obvious calculation, it follows that $[[\hat{K_1}]]\neq[[\hat{K_2}]]$. Thus $K_{1}$ and $K_{2}$ are non ambient isotopic by an isotopy which is the identity in a regular neighborhood of $J$.

\subsection{The Minimality Problem}

Let us now consider an application of our construction to establish minimality of some diagrams of classical links. Namely, we find a knot diagram $K$ having a minimum number of double points of the projection on the Seifert surface of the fibred knot $J$. Then it follows straightforwardly that there are in fact infinitely many examples of such sort.

In this subsection we establish {\itshape a new complexity} of knots diagrams~\cite{ChM}.

{\bf Definition 4}. Let $J$ be a fibred knot and $\Sigma$ be a Seifert surface of $J$. Let $K$ be a knot diagram on $\Sigma$. We will say that $K$ is {\em minimal with respect to (w.r.t.)} $J$ if for all $K'$ is ambient isotopic to $K$ in $\Sigma \times \mathbb{R}$, we have that the number of crossings of the diagram of $K$ on $\Sigma$ is less than or equal to the number of crossings of $K'$ of the projection on $\Sigma$.

The following theorem is {\em a sufficient condition of minimality} of the diagram of the knot in thickened surface and is the {\em main result} of the present paper. As it was noted above, this condition apply for any virtual diagrams unlike result of the work~\cite{ChM} which works only for irreducible odd diagram~\cite{M}.

\begin{thm}
Let $L=J \sqcup K$ be a classical two component link such that $lk(L)=0$, where $J$ is a fibred knot. Furthermore let as before $\Sigma$ be a Seifert surface of the fibred knot $J$, and $\hat{K}$ be a virtual diagram corresponding to the knot $K'$ in $\Sigma \times \mathbb{R}$, where $\pi(K')=K$. Then if the graph $K_{us}$ of the diagram of $\hat{K}$ is irreducible then $K'$ has a minimal number of double points of the projection on $\Sigma$, i.e. $\hat{K}$ is minimal w.r.t. $J$.
\end{thm}

\begin{proof}
Let $L=J \sqcup K$ be a classical two component link such that $lk(L)=0$, $J$ be a fibred knot with a unique Seifert surface $\Sigma$. Moreover, let $K'$ be a knot in $\Sigma \times \mathbb{R}$ such that $\pi(K')=K$. We know that $K'$ gives rise to the virtual knot with diagram $\hat{K}$ (see Subsection 3.1.). Then the statement of the theorem follows from Proposition 1: if a diagram $\hat{K}$ is non-minimal then its $K_{us}$-graph is reducible. Indeed, assume the contrary $[[\hat{K}]]\neq[[\hat{K''}]]$, where $\hat{K''}$ be a diagram of same virtual knot with a smaller number classical crossings, because $K_{us}$-graph of $\hat{K}$ is irreducible and non-isomorphic to the $K_{us}$-graph of $\hat{K''}$ with less number vertices.
\end{proof}

{\bf Remark 4.} We note that Theorem 6 is only a sufficient but not a necessary condition of the minimality. For example, a graph $K_{us}$ of the virtual trefoil $K$ has a bigon (see Figure \ref{ex3}). However a diagram of $K$ is a minimal diagram as has the Jones-Kauffman polynomial $X(K)=-a^{-4}-a^{-6}+a^{-10}$ has $span \langle K\rangle=6$ (see~\cite{MI}).

\begin{figure}[ht]
\centerline{\includegraphics[scale=.5]{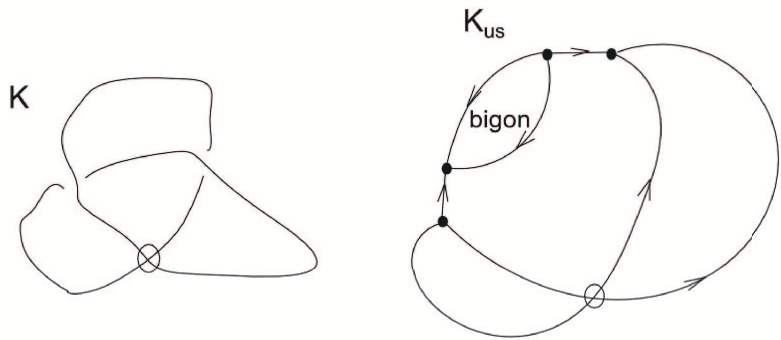}}
\caption{The virtual trefoil $K$ and its graph $K_{us}$}
\label{ex3}
\end{figure}

Recall that by {\em the Gauss diagram} corresponding to a planar diagram of a (virtual) knot we mean a diagram consisting of an oriented circle (with a fixed point, which is not a preimage of a crossing) on which the preimages of the overcrossing and the undercrossing (for each classical crossing) are connected by arrows directed from the preimage of the undercrossing to the
preimage of the overcrossing.

Let a knot and its Gauss diagram be given. We call a chord of the Gauss diagram {\em even} if the number of chords linked with it, is {\em even}, and {\em odd} otherwise (we consider a chord as unlinked with itself). We say that a Gauss diagram is {\em odd}, if all its chords are odd.

As mentioned above, Theorem 5 apply any virtual diagrams unlike result of the work~\cite{ChM} which works only for {\em irreducible odd diagram}.

Finally, we consider an example of the knot in the thickened surface with the virtual diagram which does not apply the result of the work~\cite{ChM}, but we can apply Theorem 5.

{\bf Example 2.} Let $L=J \sqcup K$ be a classical two component link such that $lk(L)=0$, $J$ be a fibred knot with a unique Seifert surface $\Sigma$ of genus 2. Furthermore, let $\hat{K}$ is non-odd virtual diagram (chords 2, 3, 5 and 6 are even) on $\Sigma$ with a irreducible $K_{us}$-graph, and $K$ be a knot in $\mathbb{S}^3\backslash J$ such that it is ''close'' to a knot diagram on $\Sigma$ which is obtained from $\hat{K}$ (see Section 1) (Fig. 10). Then by Theorem 5 $K$ is minimal w.r.t. $J$ and minimum number of double points of the projection on $\Sigma$ is equal to seven.

{\bf Acknowledgement.} The work is supported by Russian Foundation for Basic Research (grants 13-01-00830, 12-01-31507).

\begin{figure}[ht]
\centerline{\includegraphics[scale=.6]{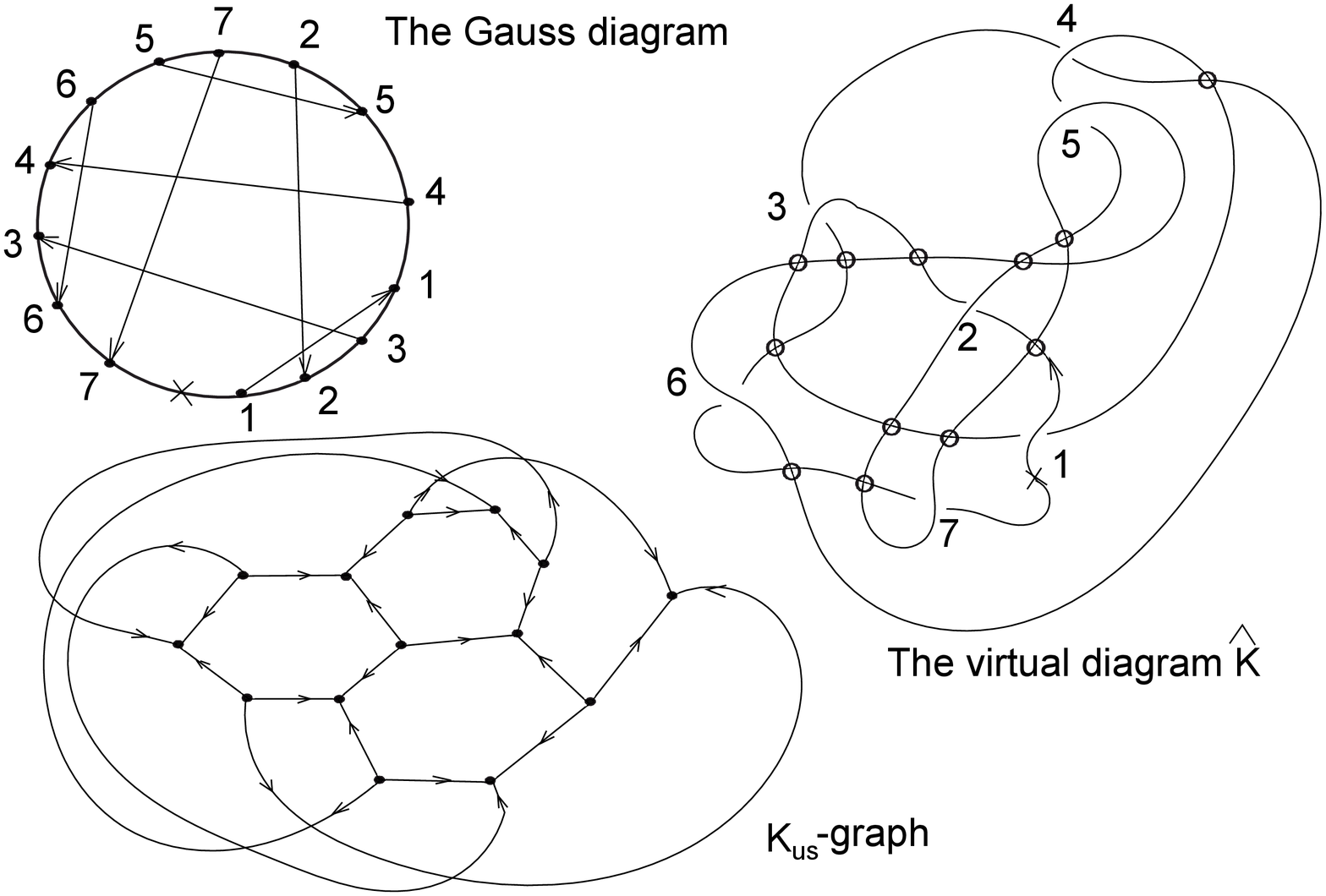}}
\caption{}
\label{exmin}
\end{figure}

\end{document}